\documentclass[3p]{elsarticle}
 \usepackage{graphics}
 \usepackage{graphicx}
 \usepackage{epsfig}
\usepackage{amssymb}
 \usepackage{amsthm}
 \usepackage{lineno}
 \usepackage{amsmath}
   \numberwithin{equation}{section}
\usepackage{mathrsfs}
\usepackage{color}

\theoremstyle{plain}
\newtheorem{theorem}{Theorem}

\newtheorem{lemma}[theorem]{Lemma}

\newtheorem{definition}[theorem]{Definition}

\numberwithin{equation}{section} \numberwithin{theorem}{section}
\newproof{pot}{{\bf{Proof of Theorem \ref{Thm1.1}}\rm}}

\begin{document}
\begin{frontmatter}
\author{Zhouji Ma}
\ead{mazj588@nenu.edu.cn}
\author{Xiaojun Chang\corref{cor2}}
\ead{changxj100@nenu.edu.cn}
\cortext[cor2]{Corresponding author.}

\address{School of Mathematics and Statistics \& Center for Mathematics and Interdisciplinary Sciences,\\
 Northeast Normal University, Changchun 130024, Jilin,
China}

\title{Normalized ground states of nonlinear biharmonic Schr\"odinger equations with Sobolev critical growth and combined nonlinearities}

\begin{abstract}
This paper is devoted to studying the following nonlinear biharmonic Schr\"odinger equation with combined power-type nonlinearities
\begin{equation*}
\begin{aligned}
   \Delta^{2}u-\lambda u=\mu|u|^{q-2}u+|u|^{4^*-2}u\quad\mathrm{in}\ \mathbb{R}^{N},
\end{aligned}
\end{equation*}
where $N\geq5$, $\mu>0$, $2<q<2+\frac{8}{N}$, $4^*=\frac{2N}{N-4}$ is the $H^2$-critical Sobolev exponent, and $\lambda$ appears as a Lagrange multiplier. By analyzing the behavior of the ground state energy with respect to the prescribed mass, we establish the existence of normalized ground state solutions. Furthermore, all ground states are proved to be local minima of the associated energy functional.
\end{abstract}
\end{frontmatter}

\section{Introduction and main results}\label{intro}
In this paper, we are concerned with the following nonlinear biharmonic Schr\"odinger equation
\begin{equation}\label{eq01}
\Delta^{2}u-\lambda u=\mu|u|^{q-2}u+|u|^{4^*-2}u\quad\mathrm{in}\ \mathbb{R}^{N}
\end{equation}
under the constraint
\begin{equation}\label{eq02}
\int_{\mathbb{R}^N}|u|^2dx=c,
\end{equation}
where $c>0$, $N\geq5$, $\mu>0$, $2<q<2+\frac{8}{N}$, $4^*=\frac{2N}{N-4}$ and $\lambda\in \mathbb{R}$ is a Lagrange multiplier.

The interest in studying \eqref{eq01}-\eqref{eq02} comes from seeking the standing wave with form $\psi(t,x)=e^{-i\lambda t}u(x)$ of the following time-dependent biharmonic Schr\"odinger equation
\begin{equation}\label{eq03}
i\partial_{t}\psi-\Delta^{2}\psi+f(|\psi|)\psi=0.
\end{equation}
This equation was considered in \cite{IK1983,T1985} to study the stability of solitons in magnetic materials once the effective quasi particle mass becomes infinite. For some recent studies on \eqref{eq03}, one can see \cite{FIB, MXZ2009, P2009, ZZ2010} etc. The biharmonic operator $\Delta^{2}$ was also used in \cite{K1996, KS2000} to reveal the effects of higher-order dispersion terms in
the mixed-dispersion fourth-order Schr\"odinger equations.

In recent years, much attention has been paid to the study of normalized solutions of nonlinear Schr\"odinger equations,
see for example \cite{J1997, JJTV2022, S2020, WW2022} and the references therein. For some related studies on the mixed-dispersion fourth-order Schr\"odinger equations, we refer the reader to \cite{BCGJ2019, BFJ2019, FJMM2022}. As far as we know, there are only very few references to the study of \eqref{eq01}-\eqref{eq02} in the literature.

Let $f(|u|)=|u|^{p-2}$. Then \eqref{eq03} is usually called $L^2$-subcritical if $p\in (2, 2+\frac{8}{N})$, $L^2$-critical if $p=2+\frac{8}{N}$, and $L^2$-supercritical if $p>2+\frac{8}{N}$. Recently, Phan \cite{Phan2018} obtained the existence of normalized ground state solutions of the biharmonic Schr\"odinger equation with a coercive potential and the $L^2$-critical nonlinearity.

In this paper, inspired by \cite{JJTV2022}, we are interested in the existence of normalized ground state solutions of \eqref{eq01}-\eqref{eq02}, where the nonlinearity consists of a $L^2$-subcritical term $\mu|u|^{q-2}u$ and the $H^2$-critical term $|u|^{4^*-2}u$.
We consider the associated energy functional
$$
J(u):=\frac{1}{2}\|\Delta u\|_{2}^{2}-\frac{\mu}{q}\|u\|_{q}^{q}-\frac{1}{4^*}\|u\|_{4^*}^{4^*}, \forall u\in H^2:=H^2(\mathbb{R}^N).
$$
By standard arguments as in \cite{M1996}, $J$ is of $C^{1}$ and any critical point of $J$ restricted to the set
$S(c):=\{u\in H^{2}: \|u\|_{2}^{2}=c\}$
corresponds to a normalized solution of \eqref{eq01}.
\begin{definition}\label{blp}
We say that a function $u_{c}\in S(c)$ is a normalized ground state solution to \eqref{eq01} if it satisfies
$J(u_{c})=\inf\{J(u), u\in S(c),(J|_{S(c)})^{'}(u)=0\}.$
\end{definition}
For $\rho>0$, set
$$\Lambda_{\rho}(c):=\{u\in S(c):\|\Delta u\|_{2}^{2}<\rho\},\qquad\partial \Lambda_{\rho}(c):=\{u\in S(c):\|\Delta u\|_{2}^{2}=\rho\}.
$$
Then we can state our main result:
\begin{theorem}\label{Thm1.1}
For any $\mu>0$, there exist $c_{0}:=c_{0}(\mu, q, N)>0$ and $\rho_0:=\rho(c_0)>0$ such that, for any $c\in (0,c_{0})$, $J|_{S(c)}$ has a ground state, which is a local minimizer of $J$ in the set $\Lambda_{\rho_0}(c)$. Furthermore, any ground state of $J|_{S(c)}$ is a local minimizer of $J$ in $\Lambda_{\rho_0}(c)$.
\end{theorem}

The paper is organized as follows. In Section \ref{preliminary},
some preliminary results are presented. In Section \ref{ground},
we give the proof of Theorem \ref{Thm1.1}.

\section{Preliminaries}\label{preliminary}
In this section, we establish some useful preliminary results. We recall that
$\|u\|_{H^{2}}:=(\int_{\mathbb{R}^{N}}|\Delta u|^{2}+|u|^{2}dx)^{\frac{1}{2}}$
is a norm on $H^{2}$ which is equivalent to the standard one
$\|u\|_{H^{2}}:=(\int_{\mathbb{R}^{N}}|D^{2}u|^{2}+|\nabla u|^{2}+|u|^{2}dx)^{\frac{1}{2}}.$  We denote by $\|\cdot\|_q$ the standard norm in $L^{q}(\mathbb{R}^N)$.

Assume that $N\geq5$. For any $u\in H^{2}$, we have the following Gagliardo-Nirenberg inequality:
\begin{equation}\label{eq04}
\|u\|_{r}\leq C_{N,r}\|\Delta u\|_{2}^{\beta }\|u\|_{2}^{(1-\beta)}~ \mathrm{for}~r\in (2,4^*),~ \beta:=\frac{N}{2}(\frac{1}{2}-\frac{1}{r}).
\end{equation}	

 By similar arguments as in \cite{M1996}, we have the following Lions' type lemma in $H^{2}$.
\begin{lemma}\label{Lions}
If $\{u_{n}\}$ is bounded in $H^{2}$ such that
$\sup\limits_{y\in\mathbb{R}^{N}}\int_{B(y,1)}|u_{n}|^{2}dx\rightarrow0\ \mbox{as}\ n\rightarrow\infty$,
then $u_{n}\rightarrow0$ in $L^{p}(\mathbb{R}^{N})$ for $2<p<4^*$.
\end{lemma}

Define
$$
f(c,\rho):=\frac{1}{2}-\frac{\mu}{q}C_{N,q}^{q}\rho^{\alpha_{0}}c^{\alpha_{1}}-\frac{\mathcal{S}^{4^*}}{4^*}\rho^{\alpha_{2}}, \forall c>0, \rho>0,
$$
where
\[
\alpha_{0}=\frac{(q-2)N}{8}-1\in(-1,0),\ \ \alpha_{1}=\frac{2N-q(N-4)}{8}\in(\frac{4}{N},1),\ \ \alpha_{2}=\frac{4}{N-4}\in(0,4],
\] and $\mathcal{S}$ is the optimal constant such that
\begin{equation}\label{eq05}
\|u\|_{4^*}\leq\mathcal{S}\|\Delta u\|_{2}.
\end{equation}
(See \cite{L1985}). Using \eqref{eq04}-\eqref{eq05}, we get
\begin{equation}\label{07-17-2}
J(u)\geq\|\Delta u\|_{2}^{2}f(c,\|\Delta u\|_{2}^{2}).
\end{equation}
By direct computations, for any $c>0$, the function $h_{c}(\rho):=f(c,\rho)$ has a unique global maximum point
\begin{equation}\label{eq08}
\rho_{c}:=[-\frac{\alpha_{0}}{\alpha_{2}}\frac{\mu C_{N,q}^{q}4^{\ast}}{q\mathcal{S}^{4^*}}]^{\frac{1}{\alpha_{2}-\alpha_{0}}}c^{\frac{\alpha_{1}}{\alpha_{2}-\alpha_{0}}}.
\end{equation}
And the maximum value is
\begin{align*}
  \max\limits_{\rho>0}h_{c}(\rho)=h_c(\rho_c)=\frac{1}{2}-Mc^{\frac{4}{N}},
\end{align*}
where
\begin{equation*}\label{eq07}
M:=\frac{\mu}{q}C_{N,q}^{q}[-\frac{\alpha_{0}}{\alpha_{2}}\frac{\mu C_{N,q}^{q}4^*}{q\mathcal{S}^{4^*}}]^{\frac{\alpha_{0}}{\alpha_{2}-\alpha_{0}}}+\frac{\mathcal{S}^{4^*}}{4^*}[-\frac{\alpha_{0}}{\alpha_{2}}\frac{\mu C_{N,q}^{q}4^*}{q\mathcal{S}^{4^*}}]^{\frac{\alpha_{2}}{\alpha_{2}-\alpha_{0}}}>0.
\end{equation*}
Specially, $f(c_{0},\rho_{c_0})=\max\limits_{\rho>0}h_{c_{0}}(\rho)=0$ for $c_{0}:=(\frac{1}{2M})^{\frac{N}{4}}>0$.

\begin{lemma}\label{compa}
Let $c_{1}>0, \rho_{1}>0$. Then, for any $c_{2}\in(0,c_{1}]$, we have
\begin{equation}\label{07-17-1}
f(c_{2},\rho)\geq f(c_{1},\rho_{1}),~~\forall \rho\in[\frac{c_{2}}{c_{1}}\rho_{1},\rho_{1}].
\end{equation}
\end{lemma}
\begin{proof}
Since $f(c, \rho)$ is decreasing with respect to $c>0$, we have $f(c_{2},\rho_{1})\geq f(c_1,\rho_{1})$. In addition, by the facts that $\alpha_{0}+\alpha_{1}=\frac{q-2}{2}$, $c_{2}\le c_{1},\ \alpha_{2}\in(0,4]$ and $2<q<\frac{8}{N}+2$, we obtain
\begin{align*}
f(c_{2},\frac{c_{2}}{c_{1}}\rho_{1})-f(c_{1},\rho_{1})=\frac{\mu}{q}C_{N,q}^{q}\rho_{1}^{\alpha_{0}}c_{1}^{\alpha_{1}}(1-(\frac{c_{2}}{c_{1}})^{\frac{q-2}{2}})+\frac{\mathcal{S}^{4^*}}{4^{\ast}}(1-(\frac{c_{2}}{c_{1}})^{\alpha_{2}})\ge0.
\end{align*}
Since the function $h_{c_{2}}(\rho)$ has a unique global maximum, we get (\ref{07-17-1}) holds.
\end{proof}
Let $u\in S(c)$ be arbitrary but fixed. Set $u_{s}(x):=s^{\frac{N}{4}}u(s^{\frac{1}{2}}x), \forall s>0$.
Then $u_{s}\in S(c)$ for any $s>0$. Set $\rho_0:=\rho_{c_0}>0$.
Define $B_{\rho_{0}}:=\{u\in H^{2}:\|\Delta u\|_{2}^{2}<\rho_{0}\}$. Clearly, $\Lambda_{\rho_0}(c):=S(c)\cap B_{\rho_{0}}$.

\begin{lemma}\label{negative gr st}
For any $c\in(0,c_{0})$, we have
\begin{equation}\label{07-17-3}
m(c):=\inf\limits_{u\in \Lambda_{\rho_0}(c)}J(u)<0<\inf\limits_{u\in\partial \Lambda_{\rho_0}(c)}J(u).
\end{equation}
\end{lemma}
\begin{proof}
For the first part of (\ref{07-17-3}), we define
$$\psi_{u}(s):=J(u_{s})=\frac{s^{2}}{2}\|\Delta u\|_{2}^{2}-\frac{\mu}{q}s^{\frac{N(q-2)}{4}}\|u\|_{q}^{q}-\frac{s^{4^*}}{4^*}\|u\|_{4^*}^{4^*}.$$
Since $\frac{N(q-2)}{4}\in(0,2)$, we deduce that $\psi_{u}(s)\rightarrow0^{-}$ as $s\rightarrow0^+$. Then, there exists $s_{0}$ small enough such that $\|\Delta u_{s_{0}}\|_{2}^{2}=s_{0}^{2}\|\Delta u\|_{2}^{2}<\rho_{0}$ and $J(u_{s_{0}})=\psi_{u}(s_{0})<0$. Hence $m(c)<0$.
	
For any $u\in \partial \Lambda_{\rho_0}(c)$, we have $\|\Delta u\|_{2}^{2}=\rho_{0}$. By the choice of $\rho_0$, we get $f(c_{0},\rho_{0})=0$ and $f(c,\rho_{0})>0$ for all $c\in(0,c_{0})$. Then, by (\ref{07-17-2}) it follows that
\begin{equation*}
J(u)\geq\|\Delta u\|_{2}^{2}f(\|u\|_{2}^{2},\|\Delta u\|_{2}^{2})=\rho_{0}f(c,\rho_{0})>0.
\end{equation*}
\end{proof}

\begin{lemma}\label{decre}
The function $c\mapsto m(c)$ is decreasing in $(0, c_0)$.
\end{lemma}
\begin{proof}
Assume that $c_1, c_2\in (0, c_0)$ and $c_2>c_1$. For any $\epsilon>0$ small enough, there exists $u\in \Lambda_{\rho_0}(c_1)$ such that
\begin{equation}\label{07-18-1}
J(u)\le m(c_1)+\epsilon~~\mbox{and}~~J(u)<0.
\end{equation}
Define $v(x):=(\frac{c_1}{c_2})^{\frac{N-4}{8}}u((\frac{c_1}{c_2})^{\frac{1}{4}}x)$, $\forall x\in \mathbb{R}^N$. Clearly, $v\in \Lambda_{\rho_0}(c_2)$. Then
\begin{eqnarray*}
m(c_2)\le J(v)=\frac{1}{2}\|\Delta u\|_2^2-\frac{\mu}{q}(\frac{c_1}{c_2})^{\frac{q(N-4)-2N}{8}}\|u\|_{q}^{q}-\frac{1}{4^*}\|u\|_{4^*}^{4^*}<J(u).
\end{eqnarray*}
Hence, together with (\ref{07-18-1}) we get $m(c_2)<m(c_1)+\epsilon$. Since $\epsilon$ is arbitrary, the conclusion follows.
\end{proof}

\begin{lemma}\label{ckj}
\begin{description}
\item[$(i)$]
The function $c\in(0, c_0)\mapsto m(c)$ is continuous.

\item[$(ii)$]Let $c\in(0,c_{0})$. Then $m(c)\leq m(\alpha)+m(c-\alpha)$ for all $\alpha\in(0,c)$. The inequality is strict if $m(\alpha)$ or $m(c-\alpha)$ is reached.
\end{description}
\end{lemma}
\begin{proof}
$(i)\ $	For any $c\in(0,c_{0})$, we take $\{c_{n}\}\subset(0,c_{0})$ be such that $c_{n}\rightarrow c$.
We first assume that $c_n>c$. By Lemma \ref{decre} we get $m(c_n)\le m(c)$. On the other hand, by Lemma \ref{negative gr st}, for any $\epsilon>0$ small enough, there exists $u_{n}\in \Lambda_{\rho_0}(c_{n})$ such that
\begin{equation}\label{eq15}
J(u_{n})\leq m(c_{n})+\epsilon\quad\mathrm{and}\quad J(u_{n})<0.
\end{equation}
Take $z_{n}=\sqrt{\frac{c}{c_{n}}}u_{n}$. Clearly, $z_{n}\in \Lambda_{\rho_0}(c)$. Then, noting that $J(z_{n})-J(u_{n})\to 0$, it follows that
$$
m(c)\leq J(z_{n})=J(u_{n})+(J(z_{n})-J(u_{n}))=J(u_{n})+o_{n}(1).
$$
By (\ref{eq15}) we get $m(c)\leq m(c_{n})+\epsilon+o_{n}(1)$. Hence $m(c_n)\to m(c)$ as $c_n\to c^+$.

In what follows, we consider the case $c_n<c$. By Lemma \ref{decre} again we get $m(c)\le m(c_n)$. For the opposite side, for any $\epsilon>0$ small, we
let $u\in \Lambda_{\rho_0}(c)$ be such that
$$J(u)\leq m(c)+\epsilon\quad\mathrm{and}\quad J(u)<0.$$
Let $u_{n}:=\sqrt{\frac{c_{n}}{c}}u$. Clearly, $u_{n}\in \Lambda_{\rho_0}(c_{n})$ for $n$ large enough, and $J(u_{n})\rightarrow J(u)$. Hence,
$$
m(c_{n})\leq J(u_{n})=J(u)+(J(u_{n})-J(u))\leq m(c)+\epsilon+o_{n}(1).
$$
Since $\epsilon>0$ is arbitrary, we obtain $m(c_n)\le m(c)+o_{n}(1)$. Thus $m(c_n)\to m(c)$ as $c_n\to c^-$.

$(ii)$  For any $\alpha\in (0, c)$, we shall prove that $m(\theta\alpha)\leq\theta m(\alpha), \forall \theta\in(1,\frac{c}{\alpha}]$.
In fact, by Lemma \ref{negative gr st}, for any $\epsilon>0$ small, there exists $u\in \Lambda_{\rho_0}(\alpha)$ such that
\begin{equation}\label{eq18}
J(u)\leq m(\alpha)+\epsilon\quad\mathrm{and}\quad J(u)<0.
\end{equation}
 Using Lemma \ref{compa}, we have $f(\alpha,\rho)\geq f(c_{0},\rho_{0})=0$ for any $\rho\in[\frac{\alpha}{c}\rho_{0},\rho_{0}]$. Then, by (\ref{07-17-2}) and  (\ref{eq18}) it follows that $\|\Delta u\|_{2}^{2}<\frac{\alpha}{c}\rho_{0}$. Take $v=\sqrt{\theta}u$ with $\theta\in(1,\frac{c}{\alpha}]$. Clearly, $v\in \Lambda_{\rho_0}(\theta\alpha)$. Then, in view of (\ref{eq18}),
\begin{align*}
m(\theta\alpha)\leq J(v)=\frac{1}{2}\theta\|\Delta u\|_{2}^{2}-\frac{\mu}{q}\theta^{\frac{q}{2}}\|u\|^{q}_{q}-\frac{1}{4^*}\theta^{\frac{4^*}{2}}\|u\|^{4^*}_{4^*}
<\theta J(u)\leq\theta(m(\alpha)+\epsilon).
\end{align*}
Since $\epsilon$ is arbitrary, we get $m(\theta\alpha)\leq\theta m(\alpha)$, which implies that
$$
m(c)=\frac{c-\alpha}{c}m(c)+\frac{\alpha}{c}m(c)=\frac{c-\alpha}{c}m(\frac{c}{c-\alpha}(c-\alpha))+\frac{\alpha}{c}m(\frac{c}{\alpha}\alpha)\leq m(c-\alpha)+m(\alpha).
$$
If $m(\alpha)$ is reached, we can choose $\epsilon=0$. Then the strict inequality follows.\end{proof}

\section{Proof of Theorem \ref{Thm1.1}}\label{ground}

In this section, we will give the proof of Theorem \ref{Thm1.1}.
Set $\mathcal{K}_{c}:=\{u\in\Lambda_{\rho_0}(c):J(u)=m(c)\}$.	

\begin{theorem}\label{com}
For any $c\in(0,c_{0})$, $\mathcal{K}_{c}\neq\O$. Furthermore, the set $\mathcal{K}_{c}$ is compact in $H^{2}$, up to translation.
\end{theorem}
\begin{proof}
Fix $c\in(0,c_{0})$. Let $\{u_{n}\}\subset B_{\rho_{0}}$ be such that $\|u_{n}\|_{2}^{2}\rightarrow c$ and $J(u_{n})\rightarrow m(c)$. Then
$J(u_n)=m(c)+o_n(1)$ and $J(u_n)<0$ for $n$ large enough. Clearly, $\{u_{n}\}$ is bounded in $H^2$. We claim that $\{u_{n}\}$ is non-vanishing. In fact, if not, we may apply Lemma \ref{Lions} to get $\|u_n\|_q\to0$. Then, by (\ref{eq04}), (\ref{eq05}) and $f(c_0, \rho_0)=0$ we get, for $n$ large enough,
\begin{eqnarray}\label{07-18-2}
0>J(u_{n})\geq\|\Delta u_{n}\|_{2}^{2}(\frac{1}{2}-\frac{\mathcal{S}^{4^{\ast}}}{4^{\ast}}\rho_0^{\alpha_{2}})+o_n(1)=\frac{\mu}{q}C_{N,q}^{q}\rho_{0}^{\alpha_{0}}c_{0}^{\alpha_{1}}\|\Delta u_{n}\|_{2}^{2}+o_n(1)>0,
\end{eqnarray}
which gives a contradiction. Then there exists a sequence $\{x_{n}\}\subset\mathbb{R}^{N}$ such that $\varliminf\limits_{n\to+\infty}\int_{B(x_{n},2)}|u_{n}|^{2}dx>0.$ Hence, up to a sequence,
 there exists a sequence $\{y_{n}\}\subset\mathbb{R}^{N}$ and $u_c \in H^2\setminus \{0\}$ such that $u_{n}(\cdot-y_{n})\rightharpoonup u_{c}$ in $H^{2}.$

Define $w_{n}(x):=u_{n}(x-y_{n})-u_{c}(x)$ for $x\in \mathbb{R}^N$. It is not difficult to see that
\begin{equation}\label{eq21}
\|w_{n}\|_{2}^{2}=\|u_{n}\|_{2}^{2}-	\|u_{c}\|_{2}^{2}+o_{n}(1),
\end{equation}
\begin{equation}\label{eq22}
\|\Delta w_{n}\|_{2}^{2}=\|\Delta u_{n}\|_{2}^{2}-\|\Delta u_{c}\|_{2}^{2}+o_{n}(1).
\end{equation}
Then, by the splitting properties of Brezis-Lieb and the translational invariance, we obtain
\begin{equation}\label{eq23}
J(u_{n})=J(u_{n}(x-y_{n}))=J(w_{n})+J(u_{c})+o_{n}(1).	
\end{equation}
Now, we claim that $\|w_{n}\|_{2}^{2}\rightarrow0$.	
Denote $\overline{c}=\|u_{c}\|_{2}^{2}>0$. We assume by contradiction that $\overline{c}<c$. In view of \eqref{eq21} and \eqref{eq22}, for $n$ large enough, we have $w_{n}\in\Lambda_{\rho_0}(\|w_{n}\|_{2}^{2})$. Then, using \eqref{eq23}, we get
$$m(c)=J(w_{n})+J(u_{c})+o_{n}(1)\geq m(\|w_{n}\|_{2}^{2})+J(u_{c})+o_{n}(1).$$
Due to \eqref{eq22}, we have $u_{c}\in \Lambda_{\rho_0}(\overline{c})$. Thus, using Lemma \ref{ckj} (i) and \eqref{eq21}, we deduce that
\begin{equation}\label{eq25}
m(c)\geq m(c-\overline{c})+J(u_{c})\ge m(c-\overline{c})+m(\overline{c}).	
\end{equation}
If $J(u_{c})>m(\overline{c})$, then it follows from Lemma \ref{ckj} (ii) and \eqref{eq25} that
$$m(c)>m(c-\overline{c})+m(\overline{c})\geq m(c-\overline{c}+\overline{c})=m(c),$$
which produces a contradiction. Hence, $J(u_{c})=m(\overline{c})$. Since $m(\overline{c})$ is reached, by \eqref{eq25} and the strict inequality in Lemma \ref{ckj} (ii), we also obtain a contradiction. Thus, the claim holds and then $\|u_{c}\|_{2}^{2}=c$.

In the following, we prove that $\|\Delta w_{n}\|_{2}^{2}\rightarrow0$. Using \eqref{eq22} we have $u_{c}\in \Lambda_{\rho_0}(c)$, which implies that $J(u_{c})\geq m(c)$. Then \eqref{eq23} and $J(u_n)\to m(c)$ imply that $J(w_{n})\leq o_{n}(1)$.
However, by \eqref{eq22} again it follows that $\{w_{n}\}\subset B_{\rho_{0}}$. Then, by $\|w_{n}\|_{2}^{2}\rightarrow0$ we get $\|w_{n}\|_{q}^{q}\rightarrow0$. Therefore, using (\ref{07-18-2}) we obtain
$\|\Delta w_{n}\|_{2}^{2}\rightarrow0$. This together with $\|w_n\|_2^2\to0$ shows that $w_{n}(x)\rightarrow0$ in $H^{2}$ and completes the proof.
\end{proof}

Define the Pohozaev set by
$$
Q_{c}:=\{u\in S(c): Q(u)=0\},
$$
where $Q(u):=\|\Delta u\|_{2}^{2}-\frac{\mu N(q-2)}{4q}\|u\|_{q}^{q}-\|u\|_{4^*}^{4^*}.$
\begin{theorem}\label{pro}
For any $c\in(0,c_{0})$, if $m(c)$ is reached, then any ground state is contained in $\Lambda_{\rho_0}(c)$.
\end{theorem}
\begin{proof}
For any $v\in S(c)$ and $s>0$, by direct calculations we get
\begin{equation}\label{eq13}
\psi_{v}^{'}(s)=\frac{1}{s}Q(v_{s}).
\end{equation}
It is well known that any critical points of $J$ stay in the set $Q_{c}$. Hence, by \eqref{eq13}, if $w\in S(c)$ is a ground state, then there exists a $v\in S(c)$ and a $s_{0}>0$ such that $w=v_{s_{0}},\ J(w)=\psi_{v}(s_{0})$ and $\psi_{v}^{'}(s_{0})=0$.

We claim that the function $\psi_{v}^{'}(s)$ has at most two zeros. It suffices to show that the function
$\xi(s):=\frac{\psi_{v}^{'}(s)}{s}$ has at most two zeros. By direct computations,
$$\xi^{'}(s)=-2\alpha_{0}\frac{\mu N(q-2)}{4q}s^{2\alpha_{0}-1}\|u\|_{q}^{q}-2\alpha_{2}s^{2\alpha_{2}-1}\|u\|_{4^*}^{4^*}.$$
In view of $\alpha_{0}<0$ and $\alpha_{2}>0$, the function $\xi^{'}(s)$ has a unique zero, which implies that the claim holds.

Since $\psi_{v}(s)\rightarrow0^{-}$ as $s\rightarrow0$, $\psi_{v}(s)\rightarrow-\infty$ as $s\rightarrow\infty$, and $\psi_{v}(s)=J(v_{s})>0$ when $v_{s}\in\partial \Lambda_{\rho_0}(c)=\{u\in S(c): \|\Delta u\|_{2}^{2}=\rho_{0}\}$, necessarily $\psi_{v}^{'}$ has a first zero $s_{1}>0$ corresponding to a local minimum, while $\psi_{v}^{'}$ has a second zero $s_{2}>0$ corresponding to a local maximum. In particular, $v_{s_{1}}\in\Lambda_{\rho_0}(c)$ and $J(v_{s_{1}})=\psi_{v}(s_{1})<0$. Hence, if $m(c)$ is reached, then it is the ground state level and the theorem comes true.
\end{proof}

\begin{pot}
By Theorem \ref{com}, $J$ admits a local minimizer in $\Lambda_{\rho_0}(c)$. Then, by Theorem \ref{pro} it follows that this local minimizer is a ground state, and any ground state of $J$ on $S(c)$ corresponds to a local minimizer in $\Lambda_{\rho_0}(c)$.
\end{pot}

\section*{Acknowledgements}
This research was supported by NSFC(11971095). This work was done when X. J. Chang visited the Laboratoire de Math\'ematiques, Universit\'e de Bourgogne Franche-Comt\'e during the period from 2021 to 2022 under the support of China Scholarship Council (202006625034), and he would like to thank the Laboratoire for their support and kind hospitality.


\section*{References}
\begin{thebibliography}{99}

\bibitem{IK1983} B. A. Ivanov, A. M. Kosevich, Stable three-dimensional small-amplitude soliton in magnetic materials,
So. J. Low Temp. Phys. 9 (1983), 439-442.

\bibitem{T1985}
S. K. Turitsyn, Three-dimensional dispersion of nonlinearity and stability of multidimensional solitons,
Teoret. Mat. Fiz. 64(2) (1985), 226-232. (in Russian)


\bibitem{FIB}
G. Fibich, B. Ilan, G. Papaniclaou, Self-focusing fourth order dispersion, SIAM J. Appl. Math. 62(4) (2002), 1437-1462.


\bibitem{MXZ2009} C. Miao, G. Xu, L. Zhao, Global well-posedness and scattering for the focusing energy-critical nonlinear Schr\"odinger equations of fourth order in the radial case, J. Differential Equations, 246(9) (2009), 3715-3749.

\bibitem{P2009} B. Pausader, The cubic fourth-order Schr\"odinger equation, J. Funct. Anal. 256(8) (2009), 2473-2517.

\bibitem{ZZ2010} J. Zhang, J. Zheng, Energy critical fourth-order Schr\"odinger equation with subcritical perturbations, Nonlinear Anal. 73(4) (2010), 1004-1014.

\bibitem{K1996}V. I. Karpman, Stabilization of soliton instabilities by higher-order dispersion: Fourth-order nonlinear Schr\"odinger-type equations, Phys. Rev. E 53(2) (1996), 1336-1339.

\bibitem{KS2000}
V. I. Karpman, A. G. Shagalov, Stability of soliton described by nonlinear Schr\"odinger-type equations
with higher-order dispersion. Phys. D 144(1-2) (2000), 194-210.


\bibitem{J1997}
L.~Jeanjean, Existence of solutions with prescribed norm for semilinear elliptic
  equations, Nonlinear Anal. 28 (10) (1997), 1633-1659.

\bibitem{JJTV2022} L. Jeanjean J. Jendrej, T. T. Le, N. Visciglia, Orbital stability of ground states for a Sobolev critical Schr\"odinger equation, J. Math. Pures Appl. 164 (2022), 158-179.


\bibitem{S2020} N. Soave, Normalized ground states for the NLS equation with combined nonlinearties: the Sobolev
    critical case, J. Funct. Anal. 279 (6) (2020), 108610.

\bibitem{WW2022}
J. C. Wei, Y. Z. Wu, Normalized solutions for Schr\"odinger equations with critical sobolev exponent and mixed nonlinearities, J. Funct. Anal. 283 (6) (2022), 109574.


\bibitem{BCGJ2019}
D. Bonheure, J.-B. Casteras, T. Gou and L. Jeanjean, Normalized solutions to the mixed dispersion nonlinear
Schr\"odinger equation in the mass critical and supercritical regime, Trans. Amer. Math. Soc. 372 (3) (2019),
2167-2212.

\bibitem{BFJ2019}
N. Boussa\"{i}d, A. J. Fern\'andez, L. Jeanjean, Some remarks on a minimization problem associated to a fourth order
nonlinear Scrh\"odinger equation, arXiv:1910.13177 (2019).

\bibitem{FJMM2022}
A. J. Fern\'andez, L. Jeanjean, R. Mandel, M. Mari\c{s}, Non-homogeneous Gagliardo-Nirenberg inequalities in $\mathbb{R}^N$ and application to a biharmonic non-linear Schr\"odinger equation, J. Differential Equations 330 (5) (2022), 1-65.

\bibitem{Phan2018} T. V. Phan, Blowup for biharmonic Schr\"odinger equation with critical nonlinearity, Z. Angew. Math. Phys. 69 (2) (2018), Paper No, 31, 11pp.


\bibitem{M1996} M. Willem, Minimax Theorems, Birkh\"aser, Boston, 1996.

\bibitem{L1985} P. L. Lions, The concentration-compactness principle in the calculus of variations. The
limit case, Parts 1, Rev. Mat. Iberoamericana 1(1) (1985), 145-201.


\end{thebibliography}
\end{document}